\documentclass[11pt,draft]{amsart}
\usepackage{amsthm,amstext,amsmath,amscd,amssymb,latexsym}
\usepackage{color}

\usepackage[matrix,arrow]{xy}
\usepackage{amsfonts,enumerate,array,hyperref}

\newcommand{\C}{{\mathbb C}}

\newcommand{\N}{{\mathbb N}}
\renewcommand{\P}{{\mathbb P}}

\newcommand{\T}{{\mathbb T}}

\newcommand{\cO}{{\mathcal O}}

\newcommand{\s}{\mathcal}

\newcommand{\sI}{{\s I}}

\newcommand{\punkt}{\HHspace{-.3ex}\raise.15ex\HHbox to1ex{\HHuge.}}

\newcommand{\paper}{: \begin{it}}
\newcommand{\jour }{, \end{it}}

\newtheorem{theorem}{Theorem}[section]
\newtheorem{lemma}[theorem]{Lemma}
\newtheorem{proposition}[theorem]{Proposition}
\newtheorem{corollary}[theorem]{Corollary}
\newtheorem{conjecture}[theorem]{Conjecture}

\theoremstyle{definition}
\newtheorem{definition}[theorem]{Definition}

\newtheorem{example}[theorem]{Example}

\theoremstyle{remark}
\newtheorem{remark}[theorem]{Remark}

\numberwithin{equation}{section}

\begin{document}

\title
{
Secant varieties of Segre-Veronese varieties $\P^m \times \P^n$ embedded by $\cO(1,2)$
 }
\author{Hirotachi Abo}
\address{Department of Mathematics, University of Idaho, Moscow, ID 83844, USA}
\email{abo@uidaho.edu}
\author{Maria Chiara Brambilla}
\address{Dipartimento di Scienze Matematiche, Universit\`a Politecnica delle Marche, Ancona, Italy}
\email{brambilla@math.unifi.it, brambilla@dipmat.univpm.it}
\date{}

\begin{abstract}
Let $X_{m,n}$ be the Segre-Veronese variety $\P^m \times \P^n$ embedded by the
morphism given by $\cO(1,2)$. In this paper, we
provide two functions $\underline{s}(m,n)\le \overline{s}(m,n)$
such that the $s^{\mathrm{th}}$ secant
variety of $X_{m,n}$ has the expected dimension if
$s \leq \underline{s}(m,n)$ or $ \overline{s}(m,n) \leq s$.
We also present a conjecturally complete list of defective
secant varieties of such Segre-Veronese varieties.
\end{abstract}
\maketitle
\section{Introduction}
\label{sec:intro}
Let $X\subset \P^N$ be an irreducible non-singular variety of
dimension $d$.  Then the $s^{\mathrm {th}}$ {\it secant variety} of $X$, denoted
$\sigma_s(X)$, is defined to be the Zariski closure of  the union of
the linear spans of all $s$-tuples of points of $X$.
The study of secant varieties has a long history.
The interest in this subject goes back to the
Italian school at the turn of the $20^{\mathrm{th}}$ century.
This topic has
received renewed interest over the past several decades,
mainly due to its increasing importance in an ever widening collection of  disciplines including
algebraic complexity theory~\cite{BCS, La2, La1}, algebraic
statistics~\cite{GSS, ERSS, AHOT}, and combinatorics~\cite{StSu, sullivant}.

The major questions surrounding secant varieties center around
finding invariants of those objects such as dimension. A  simple
dimension count suggests that the expected dimension of
$\sigma_s(X)$ is $\min \left\{N,s(d+1)-1\right\}$.  We say that
$X$ has a {\it defective $s^{\mathrm{th}}$ secant variety} if
$\sigma_s(X)$ does not have the expected dimension. In particular,
$X$ is said to be  {\it defective} if $X$ has a defective
$s^{\mathrm{th}}$ secant variety for some $s$.  For instance, the
Veronese surface $X$ in $\P^5$ is defective, because the dimension
of $\sigma_2(X)$ is four while its expected dimension is five.  A
well-known classification of the defective Veronese varieties was
completed in a series of papers by Alexander and
Hirschowitz~\cite{AH} (see also \cite{BO}). 
There are corresponding conjecturally
complete lists of defective Segre varieties~\cite{AOP} and
defective Grassmann varieties~\cite{BDG}. Secant varieties of
Segre-Veronese varieties are however less well-understood. In
recent years, considerable efforts have been made to develop
techniques to study secant varieties of such varieties (see  for
example~\cite{CGG, CaCh, CaCa, Ott, CGG1, Bal, Abr}). But  even
the classification of defective two-factor Segre-Veronese
varieties is still far from complete.

In order to classify defective Segre-Veronese varieties,
a crucial step is to prove the existence of a large family of
non-defective such varieties.
A powerful tool to establish non-defectiveness of large classes of
Segre-Veronese varieties
is the inductive approach based on specialization techniques,
which consist in placing a certain number of points on a chosen divisor.
For a given $\mathbf{n}=(n_1, \dots, n_k)  \in \N^k$, we write
$\P^{\mathbf{n}}$ for $\P^{n_1} \times \cdots \times \P^{n_k}$.
Let $X_{\mathbf{n}}^{\mathbf{a}}$ be the Segre-Veronese variety
$\P^{\mathbf{n}}$ embedded by the morphism given by
$\cO(\mathbf{a})$ with $\mathbf{a}=(a_1, \dots, a_k) \in \N^k$. As
we shall see in Section~\ref{sec:splittingtheorem},
the problem of determining the dimension of $\sigma_s(X_{\mathbf{n}}^{\mathbf{a}})$
is equivalent to the problem of determining  the value of the Hilbert function
$h_{\P^{\mathbf{n}}}(Z, \cdot)$ of a collection $Z$ of
$s$ general double points in $\P^{\mathbf{n}}$ at $\mathbf{a}$, i.e.,
\[
h_{\P^{\mathbf{n}}}(Z, \mathbf{a}) = \dim H^0 ( \P^{\mathbf{n}},
\cO(\mathbf{a})) - \dim H^0 ( \P^{\mathbf{n}}, \sI_Z(\mathbf{a})).
\]
Suppose that $a_k \geq 2$. Denote by $\mathbf{n'}$ and $\mathbf{a}'$ the $k$-tuples
$(n_1, n_2, \dots, n_k-1)$ and $(a_1, a_2, \dots, a_k-1)$ respectively.
Given a $\P^{\mathbf{n}'} \subset \P^{\mathbf{n}}$,
we have an exact sequence
\[
 0 \rightarrow \sI_{\widetilde{Z}}(\mathbf{a}') \rightarrow \sI_Z(\mathbf{a}) \rightarrow
 \sI_{Z \cap \P^{\mathbf{n'}},  \P^{\mathbf{n'}}}(\mathbf{a}) \rightarrow 0,
\]
where $\widetilde{Z}$ is the residual scheme of $Z$ with respect to $\P^{\mathbf{n}'}$.
This exact sequence gives rise to the so-called {\it Castelnuovo inequality}
\[
h_{\P^{\mathbf{n}}}(Z, \mathbf{a}) \geq h_{\P^{\mathbf{n}}}(\widetilde{Z}, \mathbf{a}')
+ h_{\P^{\mathbf{n}'}}(Z \cap \P^{\mathbf{n}'}, \mathbf{a}).
\]
Thus, we can conclude that
\begin{itemize}
\item[-] if $h_{\P^{\mathbf{n}}}(\widetilde{Z}, \mathbf{a}')$ and
$h_{\P^{\mathbf{n}'}}(Z \cap \P^{\mathbf{n}'}, \mathbf{a}')$ are the expected values and
\item[-] if the degrees of $\widetilde{Z}$ and $Z\cap \P^{\mathbf{n}'}$ are both less than or both greater than $\dim H^0(\P^{\mathbf{n}}, \cO(\mathbf{a}'))$ and $\dim H^0(\P^{\mathbf{n}'}, \cO(\mathbf{a}))$ respectively,
\end{itemize}
then $h_{\P^{\mathbf{n}}}(Z, \mathbf{a})$ is also the expected value.
By semicontinuity,  the Hilbert function of
a general collection of $s$ double points in $\P^{\mathbf{n}}$
has the expected value at $\mathbf{a}$.
This enables one to check whether or not $\sigma_s(X_{\mathbf{n}}^{\mathbf{a}})$
has the expected dimension by induction on $\mathbf{n}$ and $\mathbf{a}$.

To apply this inductive approach, we need some initial cases regarding either
dimensions or degrees.
The class of secant varieties of two-factor Segre-Veronese varieties
embedded by the morphism given by $\cO(1,2)$ can be viewed as one of such initial cases.
In fact, in this case the above-mentioned specialization technique
would involve  secant varieties of two-factor Segre varieties, most of which are known to be defective,
and thus we cannot apply this technique to find $\dim \sigma_s(X_{\mathbf{n}}^{\mathbf{a}})$
for $\mathbf{n}=(m,n)$ and $\mathbf{a}=(1,2)$.
To sidestep this problem, we therefore need an {\it ad hoc} approach.

This paper is devoted to studying secant varieties of Segre-Veronese varieties  $\P^m \times \P^n$
embedded by the morphism given by $\cO(1,2)$. Let
\[
q(m,n)= \left\lfloor\frac{(m+1){n+2 \choose 2}}{m+n+1} \right\rfloor.
\]
Our main goal is to prove the following theorem:
\begin{theorem}
Let $\mathbf{n} = (m,n)$ and let $\mathbf{a}=(1,2)$. If $n$ is
sufficiently large, then $\sigma_{s}(X_{\mathbf{n}}^{\mathbf{a}})$
has the expected dimension for $s = q(m,n)$. \label{th:mainintro}
\end{theorem}
A straightforward consequence of this theorem is the following:
\begin{corollary}
Let $\mathbf{n} = (m,n)$ and let $\mathbf{a}=(1,2)$. If $n$ is
sufficiently large, then $\sigma_{s}(X_{\mathbf{n}}^{\mathbf{a}})$
has the expected dimension for all $s \leq q(m,n)$.
\end{corollary}

In order to prove Theorem~\ref{th:mainintro},
we show that if $m \leq n+2$,  then $\sigma_{\underline{s}(m,n)}(X_{\mathbf{n}}^{\mathbf{a}})$
has the expected dimension, where
\[
 \underline{s}(m,n)=
 \left\{
 \begin{array}{ll}
 (m+1)\left\lfloor n/2 \right\rfloor-\frac{(m-2)(m+1)}{2} & \mbox{if $n$ is even;}  \\
 (m+1)\left\lfloor n/2 \right\rfloor-\frac{(m-3)(m+1)}{2} & \mbox{if $m$ and $n$ are  odd;} \\
  (m+1)\left\lfloor n/2 \right\rfloor-\frac{(m-3)(m+1)+1}{2} & \mbox{if $m$ is even and if $n$ is odd. } \
 \end{array}
 \right.
\]
Theorem~\ref{th:mainintro} then follows immediately, because $\underline{s}(m,n) =q(m,n)$ for a sufficiently large $n$ (an explicit bound for $n$ can be found just before Corollary~\ref{cor:subabundant}).

To prove that $\sigma_{\underline{s}(m,n)}(X_{\mathbf{n}}^{\mathbf{a}})$
has the expected dimension, we will use double induction on $m$ and $n$.
More precisely, we will show the following two claims:
\begin{itemize}
 \item[(i)]
Let $\mathbf{n}=(n+1,n)$. Then the secant variety
$\sigma_{\underline{s}(n+1,n)}(X_{\mathbf{n}}^{\mathbf{a}})$ has the
expected dimension. Note that the case $\mathbf{n}=(n+2,n)$ is trivial
since $\underline{s}(n+2,n)=0$.
\item[(ii)]
 Let $\mathbf{n}'=(m,n-2)$ and $\mathbf{n}=(m,n)$. If
 $\sigma_{\underline{s}(m,n-2)}(X_{\mathbf{n}'}^{\mathbf{a}})$ has the
 expected dimension, then
 $\sigma_{\underline{s}(m,n)}(X_{\mathbf{n}}^{\mathbf{a}})$ has also the
 expected dimension.
\end{itemize}

Claim (i) can be proved by an inductive approach that
specializes a certain number of points on a subvariety of $\P^m \times
\P^n$ of the form $\P^{m'} \times \P^n$ (see Section~\ref{sec:splittingtheorem} for more details).
Note that a similar approach was successfully
applied to study secant varieties of Segre varieties (see for example
\cite{AOP}).

The proof of (ii) relies on a different specialization
technique which allows to place a certain number of points on a two-codimensional
subvariety of $\P^m \times \P^n$ of the form $\P^m \times \P^{n-2}$ (see Section~\ref{sec:subabundant} for more details).
This approach can be regarded as a modification of the approach introduced in~\cite{BO} that simplifies
the proof of the Alexander-Hirschowitz theorem for cubic Veronese varieties.
We also would like to mention that the same approach
was extended to secant varieties of Grassmannians of planes in~\cite{AOP2}.

In Section~\ref{sec:superabundant}, we will modify  the above techniques
to prove the following theorem:
 \begin{theorem}
Let $\mathbf{n} = (m,n)$ and let $\mathbf{a}=(1,2)$ and let
\[
 \bar{s}(m,n)=
 \left\{
 \begin{array}{ll}
(m+1) \left\lfloor n/2 \right\rfloor +1 & \mbox{if $n$ is even;}  \\
(m+1)\left\lfloor n/2 \right\rfloor+3 & \mbox{otherwise.}
 \end{array}
 \right.
\]
Then $\sigma_s(X_{\mathbf{n}}^{\mathbf{a}})$ has the expected dimension
for any $s \geq \bar{s}(m,n)$.
 \label{th:mainintro'}
 \end{theorem}
Theorems~\ref{th:mainintro} and~\ref{th:mainintro'} complete
the classification of defective  Segre-Veronese varieties $X_{m,n}^{1,2}$ for $m=1,2$.
To be more precise,  the following is  an immediate consequence of
these theorems:
\begin{corollary}
Let $\mathbf{n}=(m,n)$ and let $\mathbf{a}=(1,2)$.
\begin{itemize}
\item[(i)] If $m=1$, then $\sigma_s(X_{\mathbf{n}}^{\mathbf{a}})$ has the expected
dimension for any $s$.
\item[(ii)]
If $m=2$, then
$\sigma_s(X_{\mathbf{n}}^{\mathbf{a}})$ has the expected dimension
unless $(n,s)=(2k+1,3k+2)$ with $k \geq 1$.
\end{itemize}
\label{cor:2n}
\end{corollary}
Note that (i) is well-known, see for example~\cite{CaCh}. We also
mention that Theorem 1.3 of \cite{BD} gives a complete
classification of the case $m=1, n=2$ for any degree
$\mathbf{a}=(d_1,d_2)$, where $d_1,d_2\geq1$. On the other hand,
to our best knowledge, (ii) was previously  unknown. The
defectiveness of  the $(3k+2)^{\mathrm{nd}}$ secant variety of
$X_{2,2k+1}^{1,2}$  has already been established (see \cite{CaCh,
Ott} for the proofs). Thus Corollary~\ref{cor:2n} (ii) completes
the classification of defective secant varieties of
$X_{2,n}^{1,2}$.

In Section~\ref{sec:conjecture},
we will give a conjecturally complete list of defective secant varieties
of $X_{m,n}^{1,2}$.  Evidence for the conjecture was provided by
results in \cite{CGG, CaCh, Ott}.
Further evidence in support of the conjecture was obtained via the
computational experiments we carried out.
Thus the first part of this section will be devoted to explaining
these experiments, which were done with the computer algebra system
{\tt Macaulay2} developed by Dan Grayson and Mike
Stillman~\cite{GS}.
The proofs of Lemmas~\ref{th:m1and2} and~\ref{th:Rbarmn} are also
based on computations in {\tt Macaulay2}.
All the {\tt Macaulay2} scripts needed to make these computations
are available at
{\tt http://www.webpages.uidaho.edu/\~{ }abo/programs.html}.
\section{Splitting Theorem}
\label{sec:splittingtheorem}
Let $V$ be an $(m+1)$-dimensional vector space over $\mathbb{C}$ and let $W$ be an $(n+1)$-dimensional vector space over $\mathbb{C}$.
For simplicity, we write $\P^{m,n}$ for   $\P^m \times \P^n=\P(V)\times \P(W)$.
In this section,  we indicate by $X_{m,n}$ for the Segre-Veronese variety
$\P^{m,n}$ embedded by the morphism $\nu_{1,d}$ given by $\cO(1,d)$ for simplicity.
Let $T_p(X_{m,n})$ be the affine cone over the tangent space  $\T_p(X_{m,n})$
to $X_{m,n}$ at a point $p \in X_{m,n}$.

For each  $p \in X_{m,n}$, there are two vectors $u \in V \setminus \{0\}$
and $v \in W \setminus \{0\}$, such that $p= [u \otimes v^d] \in  \P(V \otimes S_d(W))$.
In this way,  $p$ can be identified with $([u], [v]) \in \P^{m,n}$ through $\nu_{1,d}$.
Thus $p$ is also denoted by $([u], [v])$.
Let $p =[u \otimes v^d] \in X_{m,n}$. Then $T_p(X_{m,n}) = V \otimes v^d + u \otimes v^{d-1}W$.
We denote by  $Y_p(X_{m,n})$ (or just by $Y_p$) the $(m+1)$-dimensional subspace $V \otimes v^d$ of $V \otimes S_d(W)$.
\begin{definition}
Let $p_1, \dots, p_s, q_1, \dots, q_t$ be
general points of $X_{m,n}$ and let $U_{m,n}(s,t)$
be
the subspace of $V \otimes S_d(W)$ spanned by $\sum_{i=1}^s T_{p_i}(X_{m,n})$ and $\sum_{i=1}^t Y_{q_i}(X_{m,n})$.
Then $U_{m,n}(s,t)$
is expected to have dimension
\[
 \min \left\{ s(m+n+1)+t(m+1), (m+1){n+d \choose d} \right\}.
\]
We say that {\it $S(m,n;1,d;s;t)$ is true} if $U_{m,n}(s,t)$ has the expected dimension.
For simplicity, we denote $S(m,n;1,d;s;0)$ by $T(m,n;1,d;s)$.
\end{definition}
Note that $U_{m,n}(s,0)$ is the affine cone of $\sigma_s(X_{m,n})$.
\begin{remark}
Let $q_1, \dots, q_t$ be general points of $X_{m,n}$ and let $\sigma_s(X_{m,n})$ be the $s^{\mathrm{th}}$
secant variety of $X_{m,n}$.
By Terracini's lemma~\cite{terracini}, the span of the tangent spaces to
$X_{m,n}$ at $s$ generic points  is equal to the tangent space to $\sigma_s(X_{m,n})$
at the generic $z$ point in the linear subspace spanned by the $s$ points.
Thus the vector space $U_{m,n}(s,t)$ can be thought of as
the affine cone over the tangent space to the join $J(\P(Y_{q_1}), \dots,  \P(Y_{q_t}),\sigma_s(X_{m,n}))$
of $\P(Y_{q_1}), \dots,  \P(Y_{q_t})$ and $\sigma_s(X_{m,n})$
at a general point in the linear subspace spanned by $q_1, \dots, q_t$ and $z$.  Therefore,
$S(m,n;1,d;s;t)$ is true if and only if $J( \P(Y_{q_1}), \dots,  \P(Y_{q_t}),\sigma_s(X_{m,n}))$
has the expected dimension.
In particular, $\sigma_s(X_{m,n})$ has the expected dimension if and only if
$S(m,n;1,d;s;0)$ is true.
\label{th:terracini}
\end{remark}
\begin{remark}
Let $N=(m+1){n+d \choose d}$.
Then $H^0(\P^{m,n}, \cO(1,d))$ can be identified with the set of hyperplanes in $\P^{N-1}$.
Since the condition that a hyperplane $H \subset \P^N$ contains $\T_p(X_{m,n})$
is equivalent to the condition that $H$ intersects $X_{m,n}$ in the first infinitesimal neighborhood of $p$,
the elements of $H^0(\P^{m,n}, \sI_{p^2}(1,d))$ can be viewed as hyperplanes containing $\T_p(X_{m,n})$.
Let $q \in X_{m,n}$.  A similar argument shows that  the elements of $H^0(\P^{m,n}, \sI_{q^2|_{\P(Y_q)}}(1,d))$
can be identified with  hyperplanes containing $Y_q$, where $q^2|_{\P(Y_q)}$ is a zero-dimensional subscheme
of $X_{m,n}$ of length $m+1$.

Let $p_1, \dots, p_s, q_1, \dots, q_t \in X_{m,n}$ and
let $Z=\{p_1^2, \dots, p_s^2, q_1^2|_{\P(Y_{q_1})}, \dots, q_t^2|_{\P(Y_{q_t})}\}$.
Recall that Terracini's lemma says that
the linear subspace spanned by $\T_{p_1}(X_{m,n}), \dots, \T_{p_s}(X_{m,n})$ is the tangent
space to $\sigma_s(X_{m,n})$ at a general point in the linear subspace spanned
by $p_1, \dots, p_s$. This implies that $\dim  J(\P(Y_{q_1}), \dots,  \P(Y_{q_t}),\sigma_s(X_{m,n}))$
equals the value of  the Hilbert function
$h_{\P^{m,n}}(Z, \cdot )$ of $Z$ at $(1,d)$, i.e.,
\[
h_{\P^{m,n}}(Z,(1,d))=\dim H^0(\P^{m,n}, \cO(1,d))-
\dim H^0(\P^{m,n}, \sI_Z(1,d)).
\]
In particular,
\[
h_{\P^{m,n}}(Z,(1,d))= \min \left\{
 s(m+n+1)+t(m+1), \ N
 \right\}.
\]
if and only if $S(m,n;1,d;s;t)$ is true.
\label{th:fatPoints}
\end{remark}
\begin{definition}
\label{th:abundance}
A six-tuple $(m,n;1,d;s;t)$ is called {\it subabundant} (resp. {\it superabundant})
\[
 s(m+n+1)+t(m+1) \leq  (m+1){n+d \choose d} \ (\mbox{resp. $\geq$}).
\]
We say that  $(m,n;1,d;s;t)$  is  {\it equiabundant} if it is both subabundant and superabundant.
For brevity, we will write the five-tuple $(m,n;1,d;s)$ instead of the six-tuple $(m,n;1,d;s;0)$.
\end{definition}
Assume that $S(m,n;1,d;s;t)$ is true.
Note that when $(m,n;1,d;s;t)$ is superabundant,  $U_{m,n}(s,t)$ coincides with the whole space $V \otimes S_d(W)$, whereas
for subabundant $(m,n;1,d;s;t)$, $U_{m,n}(s,t)$ can be a proper subspace of the whole space.
\begin{remark}
Given two vectors $(s,t)$ and $(s',t')$, we say that $(s,t) \geq (s',t')$ if $s \geq s'$ and $t \geq t$.
Suppose that  $S(m,n;1,2;s;t)$ is true and that $(m,n;1,2;s;t)$ is subabundant (resp. superabundant).
Then $S(m,n;1,2;s';t')$ is true for any choice of $s'$ and $t'$ with $(s,t) \geq (s',t')$ (resp. with $(s,t) \leq (s',t')$).
\end{remark}
\begin{remark}\label{easyremark}
Suppose that $m=0$. We make the following simple remarks:
\begin{itemize}
\item[(i)] Let $q \in X_{0,m}$.
Then $\P(Y_q(X_{0,n}))$ is just $q$ itself.
If $q_1, \dots, q_t$ are general points of $X_{0,n}$
and if $(0,n;1,d;s;t)$ is subabundant, then
$S(0,n;1,d;s;t)$ is true if and only if $T(0,n;1,d;s)$ is true.
\item[(ii)] By the Alexander-Hirschowitz theorem, \cite{AH}, we know that
$T(0,n;1,d;n+1)$ is true. Then if $(0,n;1,d;s)$ is superabundant and if $s \geq n+1$, then
$T(0,n;1,d;s)$ is true.
\end{itemize}
\end{remark}
\begin{theorem}
Let $m=m'+m''+1$ and let $s=s'+s''$. If $(m',n;1,d;s';s''+t)$ and $(m'',n;1,d;s'';s'+t)$ are
subabundant (resp. superabundant, resp. equiabundant) and if
$S(m',n;1,d;s';s''+t)$ and $S(m'',n;1,d;s'';s'+t)$ are true, then $(m,n;1,d;s;t)$ is subabundant
(resp. superabundant, resp. equiabundant) and $S(m,n;1,d;s;t)$ is true.
\label{th:splitting}
\end{theorem}
\begin{proof}
Here we only prove the theorem in the case where  $(m',n;1,d;s';s''+t)$ and $(m'',n;1,d;s'';s'+t)$ are both subabundant,
because the remaining cases can be proved in a similar manner.
Let $V'$ and $V''$ be subspaces of $V$ of dimensions $m'+1$ and $m''+1$ respectively.
Suppose that $V$ is the direct sum of $V'$ and $V''$.
Let $p = [u \otimes v^d] \in X_{m,n}$. If $u \in V'$, then  we have
\begin{eqnarray*}
 T_p(X_{m,n})
 &=& V \otimes v^d + u \otimes v^{d-1}W \\
 &=&( V' \otimes v^d + u \otimes v^{d-1}W) \oplus (V'' \otimes v^d) \\
 &=& T_p (X_{m',n}) \oplus Y_{p''}(X_{m'',n})
\end{eqnarray*}
for some $p'' \in X_{m'',n}$ ($p''$ must be of the form $[u'' \otimes v^d]$ with $u'' \in V''$).
Similarly, one can prove that if $u \in V''$, then $T_p(X_{m,n}) = Y_{p'} (X_{m',n}) \oplus T_p(X_{m'',n})$
for some $p' \in X_{m',n}$.

Let $q=[u'\otimes v'^d] \in X_{m,n}$. Then there exist $q' \in X_{m',n}$ and $q'' \in X_{m'',n}$ such that
\begin{eqnarray*}
Y_q(X_{m,n})
&=&
 V \otimes v'^d \\
&=&
(V' \otimes v'^d) \oplus (V'' \otimes v'^d) \\
&=& Y_{q'}(X_{m',n}) \oplus Y_{q''}(X_{m'',n}).
\end{eqnarray*}
Thus one can conclude that $U_{m,n}(s,t) \simeq U_{m',n}(s',s''+t) \oplus U_{m'',n}(s'',s'+t)$.
By assumption,  $\dim U_{m',n}(s',s''+t) = s'(m'+n+1)+(s''+t)(m'+1)$ and $\dim U_{m'',n}(s'',s'+t) = s''(m''+n+1)+(s'+t)(m''+1)$.
Thus $\dim U_{m,n}(s,t)=\dim U_{m',n}(s',s''+t)+\dim U_{m'',n}(s'',s'+t)= s(m+n+1)+t(m+1) \leq
(m'+1){n+d \choose d}+(m''+1){n+d \choose d}=(m+1) {n+d \choose d}$, and hence
$(m,n;1,d;s,t)$ is subabundant and $S(m,n;1,d;s;t)$ is true
\end{proof}
 We will discuss three examples to illustrate how to use  Theorem~\ref{th:splitting} below.
 These examples will be used in later sections.
\begin{example}
In this example, we apply Theorem~\ref{th:splitting} to prove that $T(2,2;1,2;s)$ is true
for every $s \leq 3$.
Note that $(2,2;1,2;s)$ is subabundant for $s\leq 3$. Thus it suffices to
show that $T(2,2;1,2;3)$ is true. 
Taking $m'=1, m''=0$ and $s'=2, s''=1$,
one can reduce $ T(2,2;1,2;3)$ to $S(1,2;1,2;2;1)$ and $S(0,2;1,2;1;2)$.
Indeed $(1,2;1,2;2;1)$ and $(0,2;1,2;1;2)$ are both
subabundant.
The statement $S(1,2;1,2;2;1)$ can be reduced again to twice $S(0,2;1,2;1;2)$ by 
taking $m'=m''=0$ and $s'=s''=1$.
This means that $ T(2,2;1,2;3)$ is reduced to triple $S(0,2;1,2;1;2)$.
Clearly  $S(0,2;1,2;1;0)$ is true, and so is $S(0,2;1,2;1;2)$ by Remark \ref{easyremark} (i).
Hence we completed the proof.
\label{th:(2,2;1,2;s)}
\end{example}
\begin{example}
We prove that $T(m,1;1,2;3)$ is true for any $m$.  The proof is by induction.
It has been already proved that $T(1,1;1,2;3)$ is true (see~\cite{CGG}).
Suppose now that $T(m-1,1;1,2;3)$ is true for some $m$.
Note that $(m,1;1,2;3)$ is superabundant. Since $(m-1,1;1,2;3;0)$ and $(0,1;1,2;0;3)$ are also
superabundant, we can reduce $T(m,1;1,2;3)$ to $T(m-1,1;1,2;3)$ and $S(0,1;1,2;0;3)$.
Clearly, $S(0,1;1,2;0;3)$ is true, by Remark \ref{easyremark} (i).
Since $T(m-1,1;1,2;3)$
is true by induction hypothesis, $T(m,1;1,2;3)$ is also true.
\label{th:(m,1;1,2;3)}
\end{example}
\begin{example}
Here we prove that  $T(n+1,n;1,2;s)$ is true for any $s \leq \left\lfloor \frac{n+1}{2}\right\rfloor+1$ and any $n\geq1$.
Note that $(n+1,n;1,2;s)$ is subabundant for such an $s$. Thus it is sufficient to prove that
$T(n+1,n;1,2;s)$ is true if $s = \left\lfloor \frac{n+1}{2}\right\rfloor+1$.

First suppose that $n$ is even, i.e., $n=2k$ for some integer $k\geq1$. Then $s=k+1$.
Since $(2k,2k;1,2;k;1)$ and $(0,2k;1,2;1;k)$ are both subabundant, then
$T(2k+1,2k;1,2;k+1)$ can be reduced to $S(2k,2k;1,2;k;1)$ and
$S(0,2k;1,2;1;k)$. Analogously, $S(2k,2k;1,2;k;1)$ can be reduced to
$S(2k-1,2k;1,2;k-1;2)$ and $S(0,2k;1,2;1;k)$. That means $T(2k+1,2k;1,2;k+1)$
is now reduced to $S(2k-1,2k;1,2;k-1;2)$ and twice $S(0,2k;1,2;1;k)$
(we will denote it by $2*S(0,2k;1,2;1;k)$).
 We can repeat the same process $(k-2)$ times to reduce
 $T(2k+1,2k;1,2;k+1)$ to $S(k,2k;1,2;0;k+1)$ and $(k+1)*S(0,2k;1,2;1;k)$.
Indeed we have only to check that $(2k+1-h,2k;1,2;k+1-h;h)$ is subabundant for any $1\leq h\leq k+1$, which is true.
 Now the statement $S(k,2k;1,2;0;k+1)$ can be reduced to $S(k-1,2k;1,2;0;k+1)$ and
 $S(0,2k;1,2;0;k+1)$, since both $(k,2k;1,2;0;k+1)$ and $(k-1,2k;1,2;0;k+1)$ are subabundant.
 Analogously $S(k-1,2k;1,2;0;k+1)$ can be reduced to $S(k-2,2k;1,2;0;k+1)$ and
 $S(0,2k;1,2;0;k+1)$. Repeating the same process $k-2$ times, we can reduce
 $S(k,2k;1,2;0;k+1)$ to $(k+1)*S(0,2k;1,2;0;k+1)$.
 Recall that $(0,2k;1,2;1;k)$ and $(0,2k;1,2;0;k+1)$ are subabundant.
 Thus $S(0,2k;1,2;1;k)$ and $S(0,2k;1,2;0;k+1)$ are true, because
 $T(0,2k;1,2;1)$ and $T(0,2k;1,2;0)$ are true and by Remark \ref{easyremark} (i).
 This implies that $T(2k+1,2k;1,2;k+1)$ is true.

 In the same way, we can also prove that $T(n+1,n;1,2;s)$ is true when $n$ is odd.
 Indeed, $T(2k+2,2k+1;1,2;k+2)$ can be reduced to $(k+2)*S(0,2k+1;1,2;1;k+1)$ and
 $(k+1)*S(0,2k+1;1,2;0;k+2)$.  Since $S(0,2k+1;1,2;1;k+1)$ and $S(0,2k+1;1,2;0;k+2)$
 are true, so is $T(2k+2,2k+1;1,2;k+2)$.
 \label{th:example1}
\end{example}
As immediate consequences of Theorem~\ref{th:splitting}, we can prove the following two
propositions:
\begin{proposition}
$T(m,n;1,2;s)$ is true if $s \leq m+1$ and $m\leq {n+1 \choose 2}$ or
if $s \geq (m+1)(n+1)$.
\label{th:mn12s}
\end{proposition}
\begin{proof}
We first prove that if $m\leq {n+1 \choose 2}$, then
$T(m,n;1,2;s)$ is true for any $s \leq m+1$. Since $(m,n;1,2;s)$
is subabundant for any $s \leq m+1$, it is enough to prove that
$T(m,n;1,2;m+1)$ is true.  Applying Theorem~\ref{th:splitting}
$(m+1)$ times, we can reduce to $(m+1)*S(0,n;1,2;1;m)$. Indeed
$(0,n;1,2;1;m)$ is subabundant, since from the assumption $m\leq
{n+1 \choose 2}$ it follows
\[
 (n+1)+m \leq {n+2 \choose 2}.
\]
It also follows that $(m-h,n;1,2;m+1-h;h)$ is subabundant for
any $1\leq h\leq m-1$. Since $S(0,n;1,2;1;0)$ is true, so is
$S(0,n;1,2;1;m)$. This implies that $T(m,n;1,2;m+1)$ is true.

To show that $T(m,n;1,2;s)$ is true for any $s \geq (m+1)(n+1)$,
it is enough to prove that $T(m,n;1,2;(m+1)(n+1))$ is true, since
$(m,n;1,2;(m+1)(n+1))$ is superabundant. In the same way as in the
previous case the statement can be reduced to
$(m+1)*S(0,n;1,2;m+1;(m+1)n)$. Since $(0,n;1,2;m+1)$ is
superabundant and $T(0,n;1,2;m+1)$ is true, it follows that
$(0,n;1,2;m+1;(m+1)n)$ is superabundant and
$S(0,n;1,2;m+1;(m+1)n)$ is true. Thus $T(m,n;1,2;(m+1)(n+1))$ is
true.
\end{proof}
\begin{remark}
In Sections~\ref{sec:subabundant} and \ref{sec:superabundant},
we will use different techniques to improve the bounds  given in Proposition~\ref{th:mn12s}.
\end{remark}
\begin{proposition}
Suppose that $m \geq 1$ and $d \geq 3$. Let $\displaystyle \ell =
\left\lfloor \frac{{n+d \choose d}}{m+n+1} \right\rfloor$ and let
$\displaystyle h = \left\lceil \frac{{n+d \choose d}}{n+1}
\right\rceil$. Then
\begin{itemize}
\item[(i)]  $T(m,n;1,d;s)$ is true for any $s \leq \ell(m+1)$.
\item[(ii)] If $(n,d)\not=(2,4)$, $(3,4)$, $(4,3)$, $(4,4)$ and if $s \geq h(m+1)$, then $T(m,n;1,d;s)$ is true.
\item[(iii)] If $(n,d)=(2,4)$, $(3,4)$, $(4,3)$ or $(4,4)$, then
$T(m,n;1,d;s)$ is true for any  $s \geq (h+1)(m+1)$.
\end{itemize}
\end{proposition}
\begin{proof}
(i) Suppose that $s = \ell(m+1)$. Since
\[ \ell(n+1)+\ell m = \ell (m+n+1) \leq \frac{{n+d \choose d}}{m+n+1}
(m+n+1)= {n+d \choose d},
\]
then $(0,n;1,d;\ell;\ell m)$ is subabundant (this implies that
$(h,n;1,d;\ell+ h\ell;\ell( m - h))$ is also subabundant for all $1\leq
h \leq m$). Then $T(m,n;1,d; \ell(m+1))$ can be reduced to
$(m+1)*S(0,n;1,d;\ell;\ell m)$. Furthermore,  since $\displaystyle
\ell < \left\lfloor\frac{{n+d \choose d}}{n+1} \right\rfloor$,
then $S(0,n;1,d;\ell;0)$ is true
by the Alexander-Hirschowitz theorem.
Thus $S(0,n;1,d;\ell; \ell m)$ is true by Remark \ref{easyremark} (i).
This implies, by Theorem~\ref{th:splitting}, that
$T(m,n;1,d; \ell(m+1))$  is true.

\vspace{2mm}
\noindent
(ii) Let $s=h(m+1)$. Then $(m,n;1,d;s)$ is clearly superabundant.
The statement $T(m,n;1,d;s)$ can be reduced to $(m+1)*S(0,n;1,d;h;hm)$.
Suppose that $n \not=3, 4$. Then the Alexander-Hirschowitz theorem says that
$T(0,n;1,d;h)$ is true, and so is  $S(0,n;1,d;h;hm)$.
Hence by Theorem~\ref{th:splitting} it follows that $T(m,n;1,d;h(m+1))$ is true.

\vspace{2mm}
\noindent
(iii)
Suppose that  $(n,d)=(2,4)$, $(3,4)$, $(4,3)$ or $(4,4)$.
Then $T(0,n;1,d;h+1)$ is true by the Alexander-Hirschowitz theorem,
and thus $S(0,n;1,d;h;(h+1)m)$ is also true.
Therefore the same argument as in (ii) proves that $T(m,n;1,d;(h+1)(m+1))$ is true.
\end{proof}
\section{Segre-Veronese varieties $\P^m \times \P^n$ embedded by $\cO(1,2)$: Subabundant Case}
\label{sec:subabundant}
Let $V$ be an $(m+1)$-dimensional vector space over $\mathbb{C}$ with basis
$\{e_0, \dots, e_m\}$ and let $W$ be an $(n+1)$-dimensional vector space over $\mathbb{C}$
with basis $\{f_0, \dots, f_n\}$. As in the previous section, $X_{m,n}$ denotes $X_{m,n}^{1,2}$.
Let $U_L$ be a two-codimensional subspace of $W$ and let
$L= \P(V)\times \P(U_L)$.
Note that if $p$ is a point of $\nu_{1,2}(L)$, then
the affine cone $T_p(X_{m,n})$ over the tangent space to $X_{m,n}$ at $p$
modulo $V\otimes S_2(U_L)$
has dimension $(m+n+1)-(m+n-2+1)=2$.
\begin{definition}
Let $k = \left\lfloor n/2 \right\rfloor$ and let
\[
 \underline{s}(m,n)=
 \left\{
 \begin{array}{ll}
 (m+1)k-\frac{(m-2)(m+1)}{2} & \mbox{if $n$ is even;}  \\
 (m+1)k-\frac{(m-3)(m+1)}{2} & \mbox{if $m$ and $n$ are  odd;} \\
  (m+1)k-\frac{(m-3)(m+1)+1}{2} & \mbox{if $m$ is even and if $n$ is odd.} \
 \end{array}
 \right.
\]
Note that $\underline{s}(m,m-2)=0$.
We will sometimes drop the parameters $m,n$ when they are clear from the context.
\label{th:defn0}
\end{definition}
The goal of this section is to prove that if $m \leq n+2$, then
$T(m,n;1,2;s)$ is true for any $s \leq \underline{s}(m,n)$.
Since $(m,n;1,2;s)$ is subabundant,
it is sufficient to prove that  $T(m,n;1,2;\underline{s}(m,n))$ is true.
The key point is to restrict to subspaces of codimension $2$ and to use two-step induction
on $n$. 
It is obvious that $T(m,m-2;1,2;0)$ is true.
It also follows from  Example~\ref{th:example1} that
$T(m,m-1;1,2;\underline{s}(m,m-1))$ is true.
Thus it remains only to show that if $T(m,n-2;1,2;\underline{s}(m,n-2))$ is true,
then so is $T(m,n;1,2;\underline{s}(m,n))$. To do this, we need to
introduce the auxiliary statements $\underline{R}(m,n)$ and $Q(m,n)$
(see Definitions~\ref{th:defn} and \ref{th:defn'})
and use double induction on $m$ and $n$ to prove such auxiliary statements.
\begin{definition}
Let $k$ and $\underline{s}=\underline{s}(m,n)$ be as given in Definition~\ref{th:defn0}.
Note that
$\underline{s}(m,n-2)=\underline{s}-(m+1)$.
Let $p_1, \dots, p_{\underline{s}-(m+1)}$ be general points of  $L$, let $q_1, \dots, q_{m+1}$ be general points
of $\P^{m,n} \setminus L$
and let $\underline{V}_{m,n}$ be the vector space
$\langle V \otimes S_2(U_L), \sum_{i=1}^{\underline{s}-(m+1)} T_{p_i}(X_{m,n}), \sum_{i=1}^{m+1} T_{q_i}(X_{m,n})
\rangle$.
Note that the following inequality holds:
\begin{eqnarray*}
\dim \underline{V}_{m,n} & \leq &
(m+1){n \choose 2}+2[\underline{s}-(m+1)]+(m+1)(m+n+1) \\
& = &
\left\{
\begin{array}{ll}
(m+1){n+2 \choose 2} & \mbox{if $n$ is even, or if $m$ and $n$ are odd;} \\
(m+1){n+2 \choose 2}-1  & \mbox{if $m$ is even  and if $n$ is odd.}
\end{array}
\right.
\end{eqnarray*}
We say that {\it $\underline{R}(m,n)$ is true} if equality holds.
\label{th:defn}
\end{definition}
Remark~\ref{th:fatPoints} implies that  $\underline{R}(m,n)$ is true if and only if
\[
\dim H^0 (\P^{m,n},\sI_{Z \cup L}(1,2)) =
\left\{
\begin{array}{ll}
0 & \mbox{if $n$ is even or if $m$ and $n$ are odd;} \\
1  & \mbox{if $m$ is even  and if $n$ is odd,}
\end{array}
\right.
\]
where  $Z=\{p_1^2, \dots, p_{\underline{s}-(m+1)}^2, q_1^2, \dots, q_{m+1}^2\}$.
\begin{proposition}
Let $k$ and $\underline{s}=\underline{s}(m,n)$ be as given in Definition~\ref{th:defn0}.
If $\underline{R}(m,n)$ is true and if $T(m,n-2;1,2;\underline{s}-(m+1))$ is true, then
$T(m,n;1,2;\underline{s})$ is true.
\label{th:splitting1}
\end{proposition}
\begin{proof}
Let $p_1, \dots, p_{\underline{s}} \in \P^{m,n}$ and let  $Z = \{p_1^2, \dots, p_{\underline{s}}^2\}$. Then
it is easy to check that
\[
 \dim H^0 (\P^{m,n}, \sI_Z(1,2)) \geq
 \left\{
 \begin{array}{ll}
\displaystyle \frac{m^3-m}{2} & \mbox{if $n$ is even or if $m$ and $n$ are odd;} \\
\displaystyle k + 1 +\frac{m^3}{2} & \mbox{if $m$ is even and $n$ is odd.}
 \end{array}
 \right.
\]
Suppose that $p_1, \dots, p_{\underline{s}-(m+1)} \in L$
and that $p_{\underline{s}-m}, \dots, p_{\underline{s}} \in \P^{m,n} \setminus L$.
Let $Z = \{p_1^2, \dots, p_{\underline{s}}^2\}$.
Let $Z' = Z \cap L = \{p_1^2, \dots, p_{\underline{s}-(m+1)}^2\}$.
Then we have the following short exact sequence:
\[
 0 \rightarrow \sI_{Z \cup L} (1,2) \rightarrow
 \sI_Z (1,2) \rightarrow \sI_{Z', L}(1,2) \rightarrow 0.
\]
Taking cohomology, we have
\[
 0 \rightarrow H^0 (\P^{m,n}, \sI_{Z \cup L}(1,2)) \rightarrow
  H^0 (\P^{m,n}, \sI_{Z}(1,2)) \rightarrow H^0 (L, \sI_{Z'}(1,2)).
\]
Thus we must have
\[
 \dim H^0 (\P^{m,n}, \sI_{Z}(1,2)) \leq \dim H^0 (\P^{m,n}, \sI_{Z \cup L}(1,2))+\dim H^0 (L, \sI_{Z'}(1,2)).
\]
Since  $\underline{R}(m,n)$ and $T(m,n-2;1,2;\underline{s}-(m+1))$ are  true, we
have
\[
 \dim H^0 (\P^{m,n}, \sI_Z(1,2)) \leq
 \left\{
 \begin{array}{ll}
 \dim H^0 (L, \sI_{Z'}(1,2))+ 1  & \mbox{if $m$ is even and if $n$ is odd; }\\
 \dim H^0 (L, \sI_{Z'}(1,2)) & \mbox{otherwise, }
 \end{array}
 \right.
\]
from which the proposition follows.
\end{proof}
To prove that $T(m,n;1,2;\underline{s}(m,n))$ is true, it  is therefore enough to
show that $\underline{R}(m,n)$ is true if $m \leq n$.
The proof is again by two-step induction on $n$.
To be more precise,
we first prove $\underline{R}(m,m)$ and $\underline{R}(m,m+1)$ are true.
Then we show that if $\underline{R}(m,n-2)$ is true, then $\underline{R}(m,n)$ is also true.
\begin{proposition}\label{Rmm}
$\underline{R}(m,m)$ is true for any $m \geq 1$.
\end{proposition}
\begin{proof}
Without loss of generality, we may assume that $U_L = \langle f_2, \dots, f_{m+1} \rangle$.
Let $p_0 \dots, p_m \in \P^{m,m} \setminus L$.  For each $i \in \{0, \dots, m\}$,
we have $p_i=[u_i\otimes v_i^2]$  where $u_i\in V$ and $v_i \in W \setminus U_L$.
Recall that
\[
 T_{p_i} (X_{m,m}) = V \otimes v_i^2 + u_i \otimes v_i W.
\]
To prove the proposition, we will  find explicit vectors $u_i$'s and $v_i$'s such that
\[
V \otimes S_2(W) \equiv \sum_{i=0}^m T_{p_i}(X_{m,m}) \  \
\mbox{(mod $V \otimes S_2(U_L)$)}.
\]
Let $u_i=e_i$ for each $i \in \{0, \dots, m \}$ and let
\[
v_i = \left\{
\begin{array}{ll}
f_i  & \mbox{for $i=0, 1$,} \\
if_0+f_1+f_i & \mbox{for $2 \leq i \leq m$.}
\end{array}
\right.
\]
Then we have
\[
T_{p_i}(X_{m,m})=\left\{
\begin{array}{ll}
\left\langle e_0 \otimes f_0^2, \dots, e_m \otimes f_0^2, e_0 \otimes f_0f_1, \dots, e_0 \otimes f_0f_m  \right\rangle & \mbox{if $i=0$;} \\
\left\langle e_0\otimes f_1^2, \dots, e_m \otimes f_1^2, e_1\otimes f_0f_1, \dots, e_1 \otimes f_1f_m \right\rangle & \mbox{if $i=1$;}  \\
\left\langle e_0 \otimes (i f_0 + f_1 + f_i)^2, \dots, e_m \otimes (i f_0 + f_1 + f_i)^2,  \right. & \\
\left. e_i \otimes (if_0+f_1+f_i)f_0, \dots, e_i \otimes (if_0+f_1+f_i)f_m \right\rangle & \mbox{if $i \geq 2$.}
\end{array}
\right.
\]
Now we prove that every monomial in $\{ \ e_i \otimes f_j f_k \ | \ 0 \leq i, j \leq 1, j \leq k \leq m\ \}$
lies in $\langle V \otimes S_2(U_L), \ \sum_{i=0}^m T_{p_i}(X_{m,m}) \rangle$.

For each $i \in \{2, \dots, m\}$, we have
\begin{eqnarray*}
e_0 \otimes (if_0+f_1+f_i)^2
& \equiv &
e_0 \otimes ( i^2f_0^2 + f_1^2 +f_i^2 +2if_0f_1+2if_0f_i + 2f_1f_i) \\
& \equiv &
e_0 \otimes 2 f_1f_i
\quad (\mbox{mod $\langle V \otimes S_2(U_L), \ T_{p_1}(X_{m,m}), \ T_{p_2}(X_{m,m})\rangle$}).
\end{eqnarray*}
Indeed,
$e_0 \otimes f_0^2, e_0 \otimes f_0f_1$ and $e_0 \otimes f_0f_i$ are in $T_{p_1}(X_{m,m})$,
$e_0 \otimes f_1^2$ is in $T_{p_2}(X_{m,m})$,
$e_0 \otimes f_i^2$ is in $V \otimes S_2(U_L)$.
Similarly, one can prove that
\[
 e_1 \otimes (if_0+f_1+f_i)^2 \equiv
 e_1 \otimes 2if_0 f_i
 \quad (\mbox{mod $\langle V \otimes S_2(U_L), \ T_{p_1}(X_{m,m}), \ T_{p_2}(X_{m,m})
 \rangle$})
\]
for each $i \in \{2, \dots, m\}$.
So we have proved that $e_i \otimes f_jf_k \in \sum_{i=0}^m T_{p_i}(X_{m,m})$
if $i,j \in \{0, 1\}$ and $k \in \{0, \dots, m\}$.

Note that, for each $i \in \{2, \dots, m\}$,
\[
\begin{array}{lll}
e_i \otimes (if_0+f_1+f_i)f_0 &\equiv & i e_i \otimes f_0f_1 + e_i \otimes f_0 f_i; \\
e_i \otimes (if_0+f_1+f_i)^2 &\equiv&
2ie_i \otimes f_0f_1 + 2ie_i \otimes f_0f_i + 2e_i \otimes f_1f_i;  \\
e_i \otimes (if_0 + f_1 +f_i)f_1 & \equiv & i e_i \otimes f_0f_1 +  e_i \otimes f_1f_i
\end{array}
\]
modulo $\langle V \otimes S_2(U_L), \ \sum_{i=0}^m T_{p_i}(X_{m,m}) \rangle$.
Thus
\[
 e_i \otimes (if_0+f_1+f_i)^2 - 2e_i \otimes (if_0+f_1+f_i)f_0- ( 2 -2/i)  e_i \otimes (if_0+f_1+f_i)f_1
\]
is congruent to $(2/i) e_i \otimes f_1f_i$ modulo $\langle V \otimes S_2(U_L), \ \sum_{i=0}^m T_{p_i}(X_{m,m}) \rangle$.
Thus $e_i \otimes f_1f_i$, and hence $e_i \otimes f_0f_1$ and $e_i \otimes f_0f_i$,
is in $\langle V \otimes S_2(U_L), \ \sum_{i=0}^m T_{p_i}(X_{m,m}) \rangle$.

For every integer $j$ such that $i \not= j$ and $j \geq 2$,  we have
\[
\begin{array}{lll}
e_i \otimes (if_0+f_1+f_i)f_j & \equiv & i e_i \otimes f_0f_j + e_i \otimes f_1f_j; \\
e_i \otimes (j f_0 + f_1 +f_j)^2 & \equiv & 2j e_i \otimes f_0f_j +2e_i \otimes f_1f_j
\end{array}
\]
modulo $\langle V \otimes S_2(U_L), \ \sum_{i=0}^m T_{p_i}(X_{m,m}) \rangle$. Hence
\[
e_i \otimes (jf_0+f_1+f_j)^2 - (2j/i) e_i \otimes (if_0+f_1+f_i)f_j \equiv (2-2j/i)e_i \otimes f_1f_j.
\]
This implies that $e_i \otimes f_1f_j$, and hence $e_i \otimes f_0f_j$, is contained in
$\langle V \otimes S_2(U_L), \ \sum_{i=0}^m T_{p_i}(X_{m,m}) \rangle$, which completes the proof.
 \end{proof}
\begin{proposition}\label{Rmm+1}
$\underline{R}(m,m+1)$ is true for any $m \geq 1$.
\label{th:Rmm1}
\end{proposition}
\begin{proof}
We only prove the statement for $m$ even,
since the other case can be proved in the same way.

If $m$  is even, then $\underline{s}(m,m+1)=3m/2+1$.
Let $p_1, \dots, p_{m/2} \in L$ and let $q_1, \dots, q_{m+1} \in \P^{m,m+1} \setminus L$.
Choose a subvariety $\P^{m,m} = \P(V) \times \P(W') \subset \P^{m,m+1}$ in such a way that the intersection of $\P^{m,m}$
with $L$ is $\P^{m,m-2}$. We denote it by $H$.  Specialize $q_1, \dots, q_{m+1}$ on $H \setminus L$.
Suppose that $p_1, \dots, p_{m/2} \not\in H$.
Let $Z=\{p_1^2, \dots, p_{m/2}^2, q_1^2, \dots, q_{m+1}^2 \}$. Then we have an exact sequence
\[
 0 \rightarrow \sI_{Z \cup L \cup H}(1,2) \rightarrow
 \sI_{Z \cup L}(1,2) \rightarrow \sI_{(Z \cup L) \cap H, H}(1,2) \rightarrow 0.
\]
By Proposition~\ref{Rmm},  $\underline{R}(m,m)$ is true. Thus  $\dim H^0(\sI_{(Z \cup L) \cap H, H}(1,2))=0$.
So we have
\[
\dim  H^0(\P^{m,m+1}, \sI_{Z \cup L \cup H}(1,2)) = \dim H^0  (\P^{m,m+1}, \sI_{Z \cup L}(1,2)).
\]
Thus we need to prove that  $\dim  H^0(\P^{m,m+1}, \sI_{Z \cup L \cup H}(1,2)) =1$.

Let $\widetilde{Z}$ be the residual  of $Z \cup L$ by $H$. Then
$H^0(\P^{m,m+1}, \sI_{Z \cup L \cup H}(1,2)) \simeq  H^0
(\P^{m,m+1}, \sI_{\widetilde{Z}}(1,1))$. Note that $\widetilde{Z}$
consists of $L$, $m+1$ simple points $q_1, \dots, q_{m+1}$ and  $m/2$ double points $p_1^2, \dots, p_{m/2}^2$.

We denote by $X'_{m,m+1}$ the Segre variety $X_{m,m+1}^{1,1}$ obtained
embedding $\P^{m,m+1}$ by the morphism given by $\cO(1,1)$.
The condition that $\dim H^0 (\P^{m,m+1},
\sI_{\widetilde{Z}}(1,1))=1$, i.e., $h_{\P^{m,m+1}}(\widetilde{Z},
(1,1))= (m+1)(m+2)-1$, is equivalent to the condition that the
following subspace of $V \otimes W$ has dimension $(m+1)(m+2)-1$:
\[
\left\langle V \otimes U_L,  \ \sum_{i=1}^{m/2} T_{p_i}(X'_{m,m+1}),
\ \sum_{i=1}^{m+1} \langle u_i' \otimes v_i'\ \rangle \right\rangle,
\]
where $q_i = [u_i' \otimes v_i']$.
Without loss of generality, we may assume that
$U_L = \langle f_0, \dots, f_{m-1} \rangle$ and that
$W' = \langle f_1, \dots, f_{m+1} \rangle$.
Since $p_i \in L$
for each $i \in \{1, \dots, m/2\}$,
there are $u_i \in V$ and $v_i \in U_L$ such that $p_i = [u_i \otimes
v_i]$. Recall that 
$T_{p_i}(X_{m,m+1}') = V \otimes v_i + u_i \otimes W$. So we have 
\[
T_{p_i}(X'_{m,m+1}) \equiv u_i \otimes \langle f_m, f_{m+1}  \rangle \ \mbox{(mod $V \otimes U_L$)},
\]
which implies that
\[
\langle V \otimes U_L, T_{p_i}(X'_{m,m+1}) \rangle = ( V \otimes f_0 ) \oplus
\left\langle V \otimes (U_L \cap W'), \  \sum_{i=1}^{m/2} u_i \otimes \langle f_m, f_{m+1} \rangle \right\rangle.
\]
Thus
\begin{eqnarray*}
&& \left\langle V \otimes U_L,  \ \sum_{i=1}^{m/2} T_{p_i}(X'_{m,m+1}),
\ \sum_{i=1}^{m+1} \langle u_i' \otimes v_i'\ \rangle \right\rangle \\
&&=  (V \otimes f_0 ) \oplus \left\langle
V \otimes (U_L \cap W'), \ \sum_{i=1}^{m/2} u_i \otimes \langle f_m, f_{m+1} \rangle, \
\sum_{i=1}^{m+1} \langle u_i' \otimes v_i' \ \rangle  \right\rangle.
\end{eqnarray*}
Note that  $T_1=\{ \  e_i \otimes  f_0 \ | \ 0 \leq i \leq m \ \}$ and
$T_2=\{\ e_i \otimes f_j\  | \ 0 \leq i \leq m, 1 \leq j \leq m-1\ \}$ are bases for $V \otimes f_0$ and $V\otimes (U_L \cap W')$ respectively.
Let $u_i=e_{i-1}$ for every $i \in \left\{1, \dots, m/2\right\}$.
Then $T_3=\left\{\left. \ e_i \otimes f_j \ \right| \ 0 \leq i \leq m/2-1, m \leq j \leq m+1 \right\}$.
Let $T_4$ be the set of vectors of the standard basis for $V \otimes W$
not included in the set $T_1 \cup T_2 \cup T_3$.
Then $T_4$ consists of $m+2$ distinct non-zero vectors. Choose $m+1$ distinct elements
of $T_4$ as $u_i' \otimes v_i'$'s. Then $\bigcup_{i=1}^4 T_i$ spans a
vector space of dimension $(m+1)(m+2)-1$. Thus we completed the proof.
\end{proof}
Let  $U_M$ be another two-codimensional subspaces of $W$ and let
$M$ be the subvariety of $\P^{m,n}$ of the forms $\P(V)\times \P(U_M)$.
If $L$ and $M$ are general,  then we have
\[
 \dim H^0(\P^{m,n},\sI_{L\cup M}(1,2))=(m+1)\left[{n+2 \choose 2} - 2{n \choose 2}+{n-2 \choose 2}\right] = 4(m+1).
\]
This is equivalent to the condition that
the subspace of $V \otimes W$
spanned by $V \otimes U_L$ and $V \otimes U_M$ has codimension $4(m+1)$.
\begin{definition}
Let $p_1, \dots, p_{m+1}$ be general points of $L$ and let $q_1, \dots, q_{m+1}$
be general points of $M$. We denote by $W_{m,n}$ the subspace of $V \otimes S_2(W)$
spanned by $V \otimes S_2(U_L)$, $V \otimes S_2(U_M)$, $\sum_{i=1}^{m+1} T_{p_i}(X_{m,n})$
and  $\sum_{i=1}^{m+1} T_{q_i}(X_{m,n})$.  Then
 $\dim W_{m,n}$ is expected to be $(m+1){n+2 \choose 2}$.
 We say that {\it $Q(m,n)$ is true} if $W_{m,n}$ has the expected dimension.
 \label{th:defn'}
\end{definition}
\begin{remark}
Keeping the same notation as in the previous definition,
we denote by $Z$ the zero-dimensional subscheme
$\{ p_1^2, \dots, p_{m+1}^2, q_1^2, \dots, q_{m+1}^2\}$.
Then $Q(m,n)$ is true if and only if $\dim H^0 (\P^{m,n}, \sI_{Z\cup L \cup M}(1,2))=0$.
\end{remark}
\begin{proposition}
If $Q(m,n)$ and $\underline{R}(m,n-2)$ are true, then $\underline{R}(m,n)$ is also true.
\label{th:splitting3}
\end{proposition}
\begin{proof}
Let $\underline{s}=\underline{s}(m,n)$, 
let $p_1, \dots, p_{\underline{s}-(m+1)} \in L$ and let $q_1, \dots, q_{m+1} \in \P^{m,n} \setminus L$.
Suppose that $p_1, \dots, p_{\underline{s}-2(m+1)} \in L \cap M$, $p_{\underline{s}-2m-1} \dots,
p_{\underline{s}-(m+1)} \in L \setminus (L\cap M)$ and $q_1, \dots, q_{m+1} \in M$.
Let $Z'=Z \cap M$. Then we have an exact sequence
\[
0 \rightarrow \sI_{Z \cup L \cup M}(1,2)
\rightarrow \sI_{Z \cup L} (1,2) \rightarrow
\sI_{Z'\cup(L\cap M), M}(1,2) \rightarrow 0.
\]
Taking cohomology gives rise to the following exact sequence:
\[
0 \rightarrow H^0 (\P^{m,n}, \sI_{Z \cup L \cup M}(1,2))
\rightarrow H^0 (\P^{m,n}, \sI_{Z \cup L }(1,2))
\rightarrow
 H^0 (M, \sI_{Z'\cup(L\cap M) , M}(1,2)).
\]
By the assumption that $Q(m,n)$ is true,
we have $\dim H^0 (\P^{m,n}, \sI_{Z \cup L \cup M}(1,2)) =0$.
Thus the following inequality holds:
\[
 \dim H^0 (\P^{m,n}, \sI_{Z \cup L }(1,2))
 \leq  \dim H^0 (M, \sI_{Z' \cup(L\cap M) , M}(1,2)).
\]
Hence if $\underline{R}(m,n-2)$ is true, then  so is $\underline{R}(m,n)$.
\end{proof}
\begin{lemma}
If $Q(m-2,n)$ and $Q(1,n)$ are true, then $Q(m,n)$ is also true.
\label{th:splitting4}
\end{lemma}
\begin{proof}
Let $V'$ be a $(m-1)$-dimensional subspace of $V$ and let $V''$ be
a two-dimensional subspace of $V$. Suppose that
$V$ can be written as the direct sum of  $V' $ and $V''$.
Let $U= \langle V \otimes S_2(U_L), V \otimes S_2(U_M),T_p(X_{m,n}) \rangle$.
Suppose that $p=([u], [v]) \in \P^{m-2,n}= \P(V') \times \P(U_L)$. Then
$V \otimes v^2 \subset V \otimes S_2(W)$. Thus
\[
T_p(X_{m,n}) \equiv  T_p(X_{m-2,n}) \  (\mbox{mod $V \otimes S_2(U_L)$}).
\]
Similarly, one can prove that
\[
T_q(X_{m,n}) \equiv  T_q(X_{1,n}) \  (\mbox{mod $V \otimes S_2(U_L)$})
\]
if $q = ([u],[v]) \in \P(V'') \times \P(W)$.

This means that if $p_1, \dots, p_{m+1} \in \P(V') \times \P(W)$ and if
$q_1, \dots, q_{m+1} \in \P(V'') \times \P(W)$, then
\begin{eqnarray*}
&&\left\langle V \otimes S_2(U_L), \sum_{i=1}^{m+1}T_{p_i}(X_{m,n}),
\sum_{i=1}^{m+1}T_{q_i}(X_{m,n}) \right\rangle \\
&=& \left\langle V' \otimes S_2(U_L), \sum_{i=1}^{m+1}T_{p_i}(X_{m-2,n}) \right\rangle
\oplus  \left\langle V'' \otimes S_2(U_L), \sum_{i=1}^{m+1}T_{q_i}(X_{1,n}) \right\rangle.
\end{eqnarray*}
In other words,
$W_{m,n} \simeq W_{m-2,n} \oplus W_{1,n}$.
Thus if $Q(m-2,n)$ and $Q(1,n)$ are true, so is $Q(m,n)$.
\end{proof}
\begin{lemma}
Let $n \geq 3$.  Then $Q(1,n)$ and $Q(2,n)$ are true.
\label{th:m1and2}
\end{lemma}
\begin{proof}
Here we only prove that $Q(1,n)$ is true for any $n \geq 3$,
because the proof of the remaining case follows the same path.

Without loss of generality, we may assume that 
$U_L=\langle f_0, \dots, f_{n-2}\rangle$ and $U_M = \langle f_2, \dots, f_ n \rangle$. 
Let $U_K= \langle f_0, f_1, f_{n-1}, f_n\rangle$  
and let  $K = \P(V) \times \P(U_K)$.  
Note that  $S_2(W) = \langle S_2(U_L),  S_2(U_M),  S_2(U_K)\rangle$, and so 
\[
 V \otimes S_2(W) = \langle V\otimes S_2(U_L), V\otimes S_2(U_M), V \otimes S_2(U_K)\rangle. 
\]
In other words, $H^0 (\P^{1,n}, \sI_{L \cup M \cup K}(1,2) )=0$. 
Specializing $p_1$ and $p_2$ on $K \cap L$ 
and $q_1$ and $q_2$ on $K \cap M$ yields the following short exact sequence, where $Z=\{p_1^2,p_2^2,q_1^2,q_1^2\}$:
\[
 0 \rightarrow \sI_{Z \cup L \cup M \cup K}(1,2) 
 \rightarrow \sI_{Z \cup L \cup M}(1,2) 
 \rightarrow \sI_{(Z \cup L \cup M) \cap K, K}(1,2) 
 \rightarrow 0. 
\]
Since $H^0 (\P^{1,n}, \sI_{L \cup M \cup K}(1,2) )$ vanishes, 
so does $H^0 (\P^{1,n}, \sI_{Z \cup L \cup M \cup K}(1,2) )$. 
Thus, in order to show that $Q(1,n)$ is true,
it is enough to prove that $Q(1,3)$ is true.  

Let $p_1$ and $p_2$ be general points of $L$ and let $q_1$ and $q_2$ be general points of $M$.
To prove that $Q(1,3)$ is true, we directly show that
\begin{eqnarray}
W_{1,3} = \left\langle
V \otimes S_2(U_L), V \otimes S_2(U_M), T_{p_1}(X_{1,3}),
T_{p_2}(X_{1,3}), T_{q_1}(X_{1,3}), T_{q_2}(X_{1,3})
\right\rangle.
\label{eq:span}
\end{eqnarray}
Recall that  $T_p(X_{1,3})$ for $p=[u \otimes v^2]$ is isomorphic to
$V \otimes v^2+u \otimes v W$.
Thus we can check equality~(\ref{eq:span}) as follows.
Let $S= \C[e_0,e_1,f_0, \dots, f_3]$. Choose randomly $u_1, \dots, u_4
\in V$, $v_1, v_2 \in U_L$ and $v_3,v_4 \in U_M$. For each $i \in \{1,
\dots, 4\}$, let $T_i$ be the ideal of $S$  generated by $u_i \otimes
v_i^2, e_1 \otimes v_i^2, u_i \otimes v_i f_0, \dots, u_i
\otimes v_if_3$. Let $I_L$ and $I_M$ be the ideals of $S$ generated by
$V \otimes S_2(U_L)$ and $V \otimes S_2(U_M)$ respectively and let
$I=\sum_{i=1}^4 T_i+I_L+I_M$. The minimal set of generators for $I$ can be
computed in {\tt Macaulay2} and we checked that the members of the minimal generating set
form a basis for $V \otimes S_2(W)$.
\end{proof}
\begin{proposition}
Let $n \geq 3$. Then $Q(m,n)$ is true for any $m$.
\label{th:Qmn}
\end{proposition}
\begin{proof}
The proof is by two-step induction on $m$.
Since we have already proved this proposition for $m=1$ and $2$, we may assume that $m \geq 3$.
By induction hypothesis, $Q(m-2,n)$ is true. Since $Q(1,n)$ is true by Lemma~\ref{th:m1and2},
it immediately follows from Lemma~\ref{th:splitting4} that $Q(m,n)$ is true.
\end{proof}
As we already mentioned, the following is an immediate consequence of
Proposition~\ref{th:Qmn}:
\begin{corollary}
Let $m \leq n$. Then $\underline{R}(m,n)$ is true.
\label{th:Rmn}
\end{corollary}
\begin{proof}
The proof is by induction on $n$. By~Proposition~\ref{Rmm},
$\underline{R}(m,m)$ is true. The statement $\underline{R}(m,m+1)$
is also true by Proposition~\ref{Rmm+1}. Assume that
$\underline{R}(m,n-2)$ is true for some $n\geq m$. We may also
assume that $n\ge3$. From Proposition~\ref{th:splitting3} and
Proposition~\ref{th:Qmn} it follows, therefore, that
$\underline{R}(m,n)$ is true. Thus we have completed the proof.
\end{proof}
\begin{theorem}
Let $k$ and $\underline{s}=\underline{s}(m,n)$ be as given in Definition~\ref{th:defn0} and
suppose that $m \leq n+2$. Then $T(m,n;1,2;s)$ is true for any $s \leq \underline{s}$.
\label{th:main}
\end{theorem}
\begin{proof}
Since $(m,n;1,2;s)$ is subabundant, it is enough to prove that
$T(m,n;1,2;\underline{s})$ is true.
The proof is by induction on $n$.
If $n=m-2$, then $\underline{s}(m,m-2)=0$. Thus  $T(m,m-2;1,2;0)$ is clearly true.
If $n=m-1$,  then $\underline{s}(m,m-1)=\left\lfloor (m-1)/2\right\rfloor+1$.
By Example~\ref{th:example1},  $T(m,m-1;1,2;s)$ is true
for any $s\le \left\lfloor m/2\right\rfloor+1$.

Now suppose that the statement is true for some $m \leq n$.
By Proposition~\ref{th:splitting1},
$T(m,n;1,2;\underline{s})$ reduces to $T(m,n-2;1,2;\underline{s}-(m+1))$ and $\underline{R}(m,n)$.
By Corollary~\ref{th:Rmn}, $\underline{R}(m,n)$ is true for any
$m\leq n$. It follows therefore that $T(m,n;1,2;\underline{s})$ is true, which completes the proof.
\end{proof}
Define a function $r(m,n)$ as follows:
\[
r(m,n) =
\left\{
\begin{array}{ll}
m^3-2m & \mbox{if $m$ is even and if $n$ is odd;} \\
\displaystyle \frac{(m-2)(m+1)^2}{2} & \mbox{otherwise.} \\
\end{array}
\right.
\]
\begin{corollary}\label{cor:subabundant}
Suppose that $n > r(m,n)$. Then $T(m,n;1,2;s)$ is true if
\[s \leq
\left\lfloor\frac{(m+1){n+2 \choose 2}}{m+n+1} \right\rfloor. \]
\end{corollary}
\begin{proof}
Since $(m,n;1,2;s)$ is subabundant, it suffices to show that
$T(m,n;1,2;s)$ is true for
$s = \left\lfloor \frac{(m+1){n+2 \choose 2}}{m+n+1} \right\rfloor$.
Note that
\[
 s=
 \left\{
 \begin{array}{ll}
 (m+1)k-\frac{(m-2)(m+1)}{2} + \left\lfloor \frac{m^3-m}{2(m+n+1)}\right\rfloor & \mbox{if $n$ is even;}  \\
 (m+1)k-\frac{(m-3)(m+1)}{2} +\left\lfloor \frac{m^3-m}{2(m+n+1)}\right\rfloor   & \mbox{if $m$ and $n$ are  odd;}  \\
 (m+1)k-\frac{(m-3)(m+1)+1}{2} + \left\lfloor \frac{n+m^3+2}{2(m+n+1)}\right\rfloor   & \mbox{otherwise.}
 \end{array}
 \right.
\]
It is straightforward to show that if $n > r(m,n)$, then $s=\underline{s}(m,n)$.
Thus it follows immediately from Theorem~\ref{th:main} that $T(m,n;1,2;s)$ is true.
\end{proof}
\begin{remark}
If $m=1$, then $r(1,n)<0$. Since $\underline{s}(1,n)=n+1$, it follows that $T(1,n;1,2;n+1)$ is true.
Since $(1,n;1,2;n+1)$ is equiabundant, $T(1,n;1,2;s)$ is therefore true for any $s$.
\end{remark}
\section{Segre-Veronese varieties $\P^m \times \P^n$ embedded by $\cO(1,2)$: Superabundant Case}
\label{sec:superabundant}
In this section, we keep the same notation as in Section~\ref{sec:subabundant}.
Let $k = \left\lfloor n/2 \right\rfloor$ and let
\[
 \bar{s}(m,n)=
 \left\{
 \begin{array}{ll}
 (m+1)k+1 & \mbox{if $n$ is even;}  \\
 (m+1)k+ 3  & \mbox{otherwise.}
 \end{array}
 \right.
\]
It is straightforward to show that $(m,n;1,2;\bar{s}(m,n))$ is superabundant.
The main  goal of this section is to prove that
$T(m,n;1,2;\bar{s}(m,n))$ is true, which implies that  $T(m,n;1,2;s)$  is true for any $s \geq \bar{s}(m,n)$.
\begin{definition}
Let $\bar{s}=\bar{s}(m,n)$, 
let $p_1, \dots, p_{\bar{s}-(m+1)}$ be general points of  $L$, let $q_1, \dots, q_{m+1}$ be general points
of $\P^{m,n} \setminus L$
and let $\overline{V}_{m,n}$ be the vector space  $\left\langle V \otimes S_2(U_L), \sum_{i=1}^{\bar{s}-(m+1)} T_{p_i}(X_{m,n}), \sum_{i=1}^{m+1} T_{q_i}(X_{m,n})
\right\rangle$.
Then the following inequality holds:
\[
\dim \overline{V}_{m,n} \leq   (m+1){n+2 \choose 2}.
\]
We say that {\it $\overline{R}(m,n)$ is true} if the equality holds.
\end{definition}
\begin{remark}
In the same way as in the proofs
of Propositions~\ref{th:splitting1} and \ref{th:splitting3}, one can prove the following:
\begin{itemize}
\item[(i)] If $\overline{R}(m,n)$ and $T(m,n-2;1,2;\bar{s}(m,n-2))$ are true,
then $T(m,n;1,2;\bar{s}(m,n))$ is true.
\item[(ii)] If $Q(m,n)$ and $\overline{R}(m,n-2)$ are true, then  $\overline{R}(m,n)$ is true.  In particular,
if $\overline{R}(m,n-2)$ is true, then  $\overline{R}(m,n)$ is true, because $Q(m,n)$ is true for $n \geq 3$
by Proposition \ref{th:Qmn}. 
\end{itemize}
\label{th:twoRemarks}
\end{remark}
\begin{definition}
Suppose that $(m,n)\not=(1,1)$. A $4$-tuple $(m,n;1,d)$ is said to be {\it balanced} if
\[
 m \leq  {n+d \choose d} -n.
\]
Otherwise, we say that $(m,n;1,d)$ is {\it unbalanced}.
\label{th:unbalance}
\end{definition}
\begin{remark}\label{rem:unbalanced}
The notion of ``unbalanced" was first introduced for Segre varieties
(see for example \cite{CGG3, AOP}).
Then it was extended to Segre-Veronese varieties in
\cite{CGG1}. In the same paper it is also proved that
if $(m,n;1,d)$ is unbalanced, then $T(m,n;1,d;s)$ fails if and only if
\begin{eqnarray}
 {n+d \choose d} -n < s < \min \left\{ m+1, {n+d \choose d}\right\}.
 \label{eq:unbalanced}
\end{eqnarray}
In particular, $T(m,2;1,2;m+1)$ is true if $m \geq 5$ and $T(m,3;1,2;m+1)$ is true if $m \geq 8$.

Here we would like to briefly explain why if $s$ satisfies the above inequalities,
then $\sigma_s(X_{m,n})$ is defective.
Let $p_1, \dots, p_s$ be generic points in $X_{m,n}$.
By assumption, we have  $s < n+1$. Thus there is a proper subvariety of $\P^{m,n}$ of type
$\P^{s-1,n}$ that  contains $p_1, \dots, p_s$.  Thus we have
\begin{eqnarray*}
 \dim \sigma_s(X_{m,n}) &\leq &s(\dim \P^{m,n}-\dim \P^{s-1,n})+s {n+d \choose d} \\
 &=& s \left[ {n+d \choose d} + m+1-s\right].
\end{eqnarray*}
It is straightforward to show that if $s$ fulfills inequalities (\ref{eq:unbalanced}), then
\[
s \left[ {n+d \choose d} + m+1-s\right] < \min \left\{s(m+n+1), \ (m+1){n+d \choose d} \right\}.
\]
Thus $\sigma_s(X_{m,n})$ is defective.  This also says that, for such an $s$,
the expected dimension of $\sigma_s(X_{m,n})$ is $s \left[ {n+d \choose d} + m+1-s\right]$.
\end{remark}
\begin{lemma}

$\mathrm{(i)}$ If $m \geq 3$, then $\overline{R}(m,n)$ is true for any $n \geq 2$;

$\mathrm{(ii)}$ $\overline{R}(2,n)$ is true for any $n \geq 3$.
\label{th:Rbarmn}
\end{lemma}
\begin{proof}
We first prove (i) for $m \geq 8$.
By Proposition \ref{th:Qmn} and Remark~\ref{th:twoRemarks} (ii), it suffices to show that $\overline{R}(m,2)$ and $\overline{R}(m,3)$ are true for any $m\ge 8$.
Suppose that $n \in \{2,3\}$.
If $m \geq 8$, then $(m,n;1,2)$ is unbalanced. Furthermore,
$(m,n;1,2;m+1)$ is superabundant. Thus $\overline{R}(m,n)$ can be reduced to $T(m,n;1,2;m+1)$.
By Remark \ref{rem:unbalanced}, $T(m,n;1,2;m+1)$ is true
for $n \in \{2,3\}$. Thus $\overline{R}(m,n)$ is also true  for $m \geq 8$ and $n \in \{2,3\}$.

The remaining cases of (i) can be checked directly as follows:
Let $S=\C[e_0, \dots, e_m, f_0, \dots, f_n]$ and let $\bar{s}=\bar{s}(m,n)$.
Choose randomly $u_1, \dots, u_{\bar{s}} \in V$, $v_1, \dots, v_{\bar{s}-(m+1)} \in U_L$ and
$v_{\bar{s}-m}, \dots, v_{\bar{s}} \in W$. For each $i \in \{1, \dots, \bar{s}\}$, let $T_i$ be the ideal
of $S$ generated by $e_0 \otimes v_i^2, \dots, e_m \otimes v_i^2$, $u_i \otimes v_if_0, \dots,
u_i \otimes v_i f_n$ and let $I_L$ be the ideal generated by $V \otimes S_2(U_L)$.
Let $I=\sum_{i=1}^{\bar{s}}T_i+I_L$. Computing the minimal generating set of $I$, we can check in {\tt Macaulay2} that the vector space spanned by homogeneous elements of $I$ of
the multi-degree $(1,2)$  coincides with  $V \otimes S_2(W)$.

Claim (ii) can be also checked in the same way.
\end{proof}
\begin{theorem}
$T(m,n;1,2;s)$ is true for any $s \geq \bar{s}(m,n)$.
\label{th:mainsuper}
\end{theorem}
\begin{proof}
In Example~\ref{th:(m,1;1,2;3)}, we showed that $T(m,1;1,2;3)$ is true
for any $m$. One can directly check that $T(2,2;1,2;4)$ is true.
So, since $\overline{R}(2,n)$ is true for any $n \geq 3$ by
Proposition~\ref{th:Rbarmn}, it follows from
Remark~\ref{th:twoRemarks} (i) that $T(2,n;1,2;\overline{s})$ is true
for any $n \geq 1$.

Suppose now that $m \geq 3$.
If $n=0$, then $\bar{s}(m,0)=1$, and obviously $T(m,0;1,2;1)$ is true.
If $n=1$, $T(m,1;1,2;3)$ is true.
Moreover by Proposition~\ref{th:Rbarmn}, we know that
$\overline{R}(m,n)$ is true for any $n \geq 2$. Hence,  from
Remark~\ref{th:twoRemarks} (i) it follows that $T(m,n;1,2;\bar{s}(m,n))$ is
true for any $n$ and any $m\geq3$.
This concludes the proof.
\end{proof}
\section{Conjecture}
\label{sec:conjecture}
Let $X_{m,n}$ be the Segre-Veronese variety $\P^{m,n}$ embedded by the morphism given by $\cO(1,2)$.
The main purpose of this section is to give a conjecturally complete list of defective secant varieties
of $X_{m,n}$.

Let $V$ be an $m$-dimensional vector space over $\C$ with basis $\{e_0, \dots, e_m\}$
and let $W$ be an $n$-dimensional vector space
over $\C$ with basis $\{f_0, \dots, f_n\}$.  As mentioned at the beginning of Section
~\ref{sec:splittingtheorem}, for a given point $p=[u \otimes v^2] \in X_{m,n}$, the affine cone
$T_p(X_{m,n})$
over the tangent space to $X_{m,n}$ at $p$ is isomorphic to $V \otimes v^2 + u \otimes v W$.
Let $A(p)$ be the $(m+1) \times (m+1){n+2 \choose 2}$ matrix whose $i^{\mathrm{th}}$
row corresponds to $e_i \otimes v^2$ and
let $B(p)$ be the
$(n+1) \times (m+1){n+2 \choose 2}$ matrix whose $i^{\mathrm{th}}$ row corresponds to $u \otimes vf_i$.
Then $T_p(X_{m,n})$ is represented by the $(m+n+2) \times (m+1){n+2 \choose 2}$ matrix $C(p)$
obtained by stacking $A(p)$ and $B(p)$:
\[
 C(p)=(\ A(p)\ ||\ B(p)\ ).
\]
For randomly chosen points $p_1, \dots, p_s \in X_{m,n}$,
let $T_s(X_{m,n})= \sum_{i=1}^s T_{p_i}(X_{m,n})$.
Then $T_s(X_{m,n})$ is represented by the $s(m+n+2) \times (m+1){n+2 \choose 2}$ matrix
$C(p_1, \dots, p_s)$
defined by
\[
C(p_1, \dots, p_s) = (\ C(p_1) \ || \ C(p_2) \ || \ \cdots \ || \ C(p_s) \ ).
\]
Thus Remark~\ref{th:terracini} and semicontinuity  imply that if
\[
 \mathrm{rank} \ C(p_1, \dots, p_s) = \min \left\{ s(m+n+1), \ (m+1) {n+2 \choose 2} \right\},
\]
then $\sigma_s(X_{m,n})$ has the expected dimension.

We programed this in {\tt Macaulay2} and computed the dimension of
$\sigma_s(X_{m,n})$ for $m,n \leq 10$ to detect ``potential" defective
secant varieties of $X_{m,n}$.
This experiment shows that $X_{m,n}$ is non-defective except for
\begin{itemize}
\item $(m,n;1,2)$ unbalanced;
\item $(m,n)=(2,n)$, where $n$ is odd and $n \leq 10$;
\item $(m,n)=(4,3)$.
\end{itemize}
\begin{remark}
The defective cases we found in the experiments are all well-known.
In Remark~\ref{rem:unbalanced}, we gave an explanation of why if $(m,n;1,2)$ is unbalanced,
then $X_{m,n}$ is defective.
Here we will discuss the remaining known defective cases.
\begin{itemize}
\item[(i)]
It is classically known that
$\sigma_5(X_{2,3})$ is defective (see~\cite{CaCh} and \cite{CaCa} for  modern proofs).
Carlini an Chipalkatti  proved in their work  on Waring's problem for
several algebraic forms~\cite{CaCh} that
$T(2,5;1,2;8)$ is false.
In~\cite{Ott}, Ottaviani then proved, as a generalization of the Strassen theorem
\cite{Stra} on three-factor Segre varieties,
that $T(2,n;1,2;s)$ fails if $(n,s)=(2k+1,3k+2)$ for any $k \geq 1$.  Here we sketch his proof
for the defectiveness of $X_{2,2k+1}$.
Recall that $X_{2,2k+1}$ is the image of the Segre-Veronese embedding
\[
\nu_{1,2}:\P(V)\times \P(W)\to \P(V\otimes S^2W),
\]
where $V$ and $W$ have dimension $3$ and  $2k+2$ respectively.
For every tensor $\phi\in V\otimes S^2W$, let $S_\phi: V\otimes W^\vee \to \wedge^2
V\otimes W\cong V^\vee\otimes W$ be the contraction operator induced by $\phi$.
If $P, Q$ and $R$ are the three symmetric slices of $\phi$,
then $S_\phi$ can be written as a skew-symmetric matrix of order $3(2k+2)$
of the form
\[
S_\phi=
\left[\begin{array}{ccc}
0&P&Q\\
-P&0&R\\
-Q&-R&0
\end{array}
 \right].
\]
The rank of  $S_\phi$ is $3(2k+2)$ for a general tensor $\phi\in V\otimes S^2W$.
On the other hand, since the contraction operator corresponding to a
decomposable tensor has rank $2$, we have $\mathrm{rank}\  S_\phi\leq
2s$, if $\phi$ is the sum of $s$ decomposable tensors.
Since the decomposable tensors correspond to the points of  the
Segre-Veronese variety, we can deduce that if $s=3k+2$, then
$\sigma_s(X_{2,2k+1})$ does not fill $\P^{3{2k+3 \choose 2}-1}$.
\item[(ii)]
The defectiveness of $\sigma_6(X_{4,3})$ can be proved by the existence of a certain
rational normal curve in $X_{4,3}$ passing through generic six points of $X_{4,3}$.
Let $\pi_1 : \P^{4,3} \rightarrow \P^4$ and $\pi_2: \P^{4,3} \rightarrow \P^3$
be the canonical projections. Given generic points $p_1, \dots, p_6 \in \P^{4,3}$,
there is a unique twisted cubic $\nu_3: \P^1 \rightarrow C_3 \subset \P^3$ that passes through
$\pi_2(p_1), \dots, \pi_2(p_6)$. Let $q_i =\nu_3^{-1}(\pi_2(p_i))$ for each $i \in \{1, \dots, 6\}$.
Since any ordered subset of six points in general position in $\P^4$ is projectively equivalent
to the ordered set $\{\pi_1(p_1), \dots, \pi_1(p_6)\}$, there is a rational quartic curve
$\nu_4: \P^1 \rightarrow C_4 \subset \P^4$ such that
$\nu_4 (q_i)=\pi_1(p_i)$ for all $i \in \{1, \dots, 6\}$. Let $\nu=(\nu_4, \nu_3)$ and let $C = \nu(\P^1)$.
Then $C$ passes through $p_1, \dots, p_6$. The image of $C$ under the morphism given
by $\cO(1,2)$ is  a rational normal curve of degree $10(=4 \cdot 1+ 2 \cdot 3)$ in $\P^{10}$.
Thus we have
\[
 \dim \sigma_6(X_{4,3}) \leq 10 +6(7-1)=46 < 6(4+3+1)-1=47,
\]
and so  $\sigma_6(X_{4,3})$ is defective.
See~\cite{CaCh} for an alternative proof.
\end{itemize}
\label{rem:defective}
\end{remark}
The experiments with our program and Remark~\ref{rem:defective} suggest the following conjecture:
\begin{conjecture}
Let $X_{m,n}$ be the Segre-Veronese variety $\P^{m,n}$ embedded by the morphism given
by $\cO(1,2)$. Then $\sigma_s(X_{m,n})$ is defective if and only if
$(m,n,s)$ falls into one of the following cases:
\begin{itemize}
\item[(a)] $(m,n;1,2)$ is unbalanced and ${n+2 \choose 2}-n < s < \min\left\{ m+1, \ {n+2 \choose 2}\right\}$;
\item[(b)] $(m,n,s)=(2,2k+1,3k+2)$ with $k \geq 1$;
\item[(c)] $(m,n,s)=(4,3,6)$.
\end{itemize}
\end{conjecture}
It is known that the conjecture is true for $m=1$ (see~\cite{CaCh}).
Here we prove that the conjecture is true for $m=2$ as a consequence of Theorems~\ref{th:main} and \ref{th:mainsuper}.
\begin{theorem}
$T(2,n;1,2;s)$ is true for any $s$ except $(n,s)=(2k+1,3k+2)$ with $k \geq 1$.
\end{theorem}
\begin{proof}
Assume first that $n=2k$ is even. Then we have $\overline{s}=\underline{s}=3k+1$.
Hence, from Theorems~\ref{th:main} and \ref{th:mainsuper}, it follows that
$T(2,2k;1,2;s)$ is true for any $s$.

Suppose now that  $n=2k+1$ is odd. Then we have
$\underline{s}=3k+1$ and $\overline{s}=3k+3$.
Thus $T(2,n;1,2;s)$ is true for any $s \leq 3k+1$, by Theorem~\ref{th:main},
and for any $s \geq 3k+3$, by Theorem~\ref{th:mainsuper}.

If $n=1$, then $\underline{s}=1$ and $\overline{s}=3$.
So it remains only to prove that also $T(2,1;1,2;2)$ is true.
But this has been already proved in Example \ref{th:example1}.
So we completed the proof.
\end{proof}
In~\cite{Ott} it is also claimed that  $\sigma_{3k+2}(X_{2,2k+1})$ is a
hypersurface if $k \geq 1$ and that this can be proved  by modifying
Strassen's argument in~\cite{Stra}.
Then it follows that the equation of $\sigma_{3k+2}(X_{2,2k+1})$ is
given by the pfaffian of $S_{\phi}$, where $S_\phi$ is the skew-symmetric
matrix introduced in Remark~\ref{rem:defective} (i).
We conclude this paper by giving an alternative proof of the fact that
$\sigma_{3k+2}(X_{2,2k+1})$ is a hypersurface for $k \geq 1$.
\begin{definition}
Suppose that $n$ is odd.
Let $s=3\left\lfloor n/2 \right\rfloor+2$,  let $p_1, \dots,
p_{s-3}$ be general points of  $L$, let $q_1, q_2 , q_3$ be general points
of $\P^{2,n} \setminus L$
and let $V_{2,n}$ be the vector space  $\left\langle V \otimes S_2(U_L), \sum_{i=1}^{s-3} T_{p_i}(X_{2,n}), \sum_{i=1}^{3} T_{q_i}(X_{2,n})
\right\rangle$.
Then the following inequality holds:
\[
\dim V_{2,n} \leq   3{n+2 \choose 2}.
\]
We say that {\it $R(2,n)$ is true} if the equality holds.
\end{definition}
\begin{lemma}
Let $n\geq 3$ be an odd integer. Then $R(2,n)$ is true.
\label{th:R2n}
\end{lemma}
\begin{proof}
The proof is very similar to that of Proposition~\ref{th:Rmm1}. One can easily prove that
if $Q(2,n)$ is true and if $R(2,n-2)$ is true, then $R(2,n)$ is true.
Since we have already proved that $Q(2,n)$ is true, it suffices to show that $R(2,3)$ is true.

Let $p_1, p_2 \in L$ and let $q_1, q_2, q_3 \in \P^{2,3}$.  Choose a subvariety  $H$  of $\P^{2,3}$
of the form $\P^{2,2} = \P(V) \times \P(W')$ such that  $\P^{2,2}$ intersects $L$ in $\P^{2,0}$.
Suppose that $p_1, p_2 \not\in H$. Specializing $q_1, q_2$ and $q_3$ in $H \setminus L$,
we obtain an exact sequence:
\[
 0 \rightarrow \sI_{Z \cup L \cup H}(1,2)
 \rightarrow \sI_{Z \cup L}(1,2)
 \rightarrow \sI_{(Z \cup L) \cap H, H}(1,2)
 \rightarrow 0,
\]
where $Z = \{p_1^2, p_2^2, q_1^2, q_2^2, q_3^2\}$.  Since we have
already proved that $\underline{R}(2,2)$ is true, we can conclude that
$\dim H^0 (\sI_{(Z \cup L)\cap H, H}(1,2))=0$.
Thus it is enough to prove that $H^0 (\sI_{Z \cup L \cup H}(1,2))=0$ or
$H^0 (\sI_{\widetilde{Z}}(1,1))=0$, where $\widetilde{Z}$ is the
residual of $Z \cup L$ by $H$. Note that $\widetilde{Z}$ consists of
two double points $p_1^2$, $p_2^2$, three simple points $q_1, q_2,
q_3$ in $H$ and $L$.
Let $X'_{2,3}$ be the Segre-Veronese variety $\P^{2,3}$ embedded by
$\cO(1,1)$. We want to prove that $L$, $\sum_{i=1}^2T_{p_i}(X'_{2,3})$
and $\sum_{i=1}^3T_{q_i}(X'_{2,3})$ span $V \otimes W$.
Note that if $p=[u \otimes v]$, then $T_p(X'_{2,3})= V \otimes v + u
\otimes W$. Now assume the following:
\begin{itemize}
 \item $U_L = \langle f_0, f_1 \rangle$ and $W' = \langle f_1, f_2, f_3 \rangle$ ;
 \item $p_1=e_0 \otimes f_0, p_2= e_1 \otimes f_1 \in V \otimes U_L$;
 \item $q_1=e_2 \otimes f_2, q_2=e_2 \otimes f_3 \in V \otimes W'$.
\end{itemize}
For any non-zero $q_3 \in V \otimes W'$,
one can show that
\[
V \otimes W = \left\langle L, \sum_{i=1}^2T_{p_i}(X'_{2,3}), \sum_{i=1}^3T_{q_i}(X'_{2,3}) \right\rangle.
\]
Thus we complete the proof.
\end{proof}
\begin{proposition}
If $(n,s)=(2k+1,3k+2)$ for $k \geq 1$, then
$\dim \sigma_s(X_{2,n}) = 3{n+2 \choose 2}-2$.
\label{th:hypersurface}
\end{proposition}
\begin{proof}
The proof is by induction on $k$.
It is well-known that $\sigma_5(X_{2,3})$ is a hypersurface.
Thus we may assume that $k \geq 2$. Let $p_1, \dots, p_s \in \P^{2,n}$.
Then there is a subvariety $L$ of $\P^{2,n}$ of the form $\P^{2,n-2}$ such that
$p_1,p_2,p_3 \in L$.  Let us suppose that $p_4, \dots, p_s \in \P^{2.n} \setminus L$.
Then we have an exact sequence
\[
 0 \rightarrow \sI_{Z \cup L}(1,2) \rightarrow \sI_Z (1,2)
 \rightarrow \sI_{Z \cap L, L}(1,2) \rightarrow 0.
\]
Taking cohomology, we get
\[
 \dim H^0 (\sI_Z (1,2))  \leq  \dim H^0 (\sI_{Z \cup L}(1,2) ) + \dim H^0 (\sI_{Z \cap L, L}(1,2)).
\]
By Lemma~\ref{th:R2n},   $\dim H^0 (\sI_{Z \cup L}(1,2) )=0$. Thus, by induction hypothesis,
\[
\dim H^0 (\sI_Z (1,2))  \leq   \dim H^0 (\sI_{Z \cap L, L}(1,2)) \leq 1.
\]
As already claimed, it is known that $T(2,n;1,2;s)$ does not
hold, i.e. $\dim H^0 (\sI_Z (1,2))  \geq 1$.
It follows that $\sigma_{s}(X_{2,n})$ is a hypersurface in
the ambient space
$\P^{3\binom{n+2}{2}-1}$.
\end{proof}
\noindent
{\it Acknowledgements.}
We thank Tony Geramita and Chris Peterson for organizing the Special Session on ``Secant Varieties
and Related Topics" at the Joint Mathematics Meetings in San Diego, January 2008.
Our collaboration started in the stimulating atmosphere of that meeting.
We also thank Giorgio Ottaviani for useful suggestions and
Enrico Carlini for his communications regarding this paper and a careful reading.
The first author would like to thank FY 2008 Seed Grant Program of the University of Idaho Research Office.
The second author (partially supported by Italian MIUR) would like to thank the University of Idaho
for the warm hospitality and for financial support.

\end{document}